\documentclass[12pt]{amsart}

\usepackage{amsmath}
\usepackage{amsthm}
\usepackage{amsfonts}
\usepackage{amssymb}
\usepackage{amscd}

\usepackage[dvips]{graphicx}
\usepackage{psfrag}

\theoremstyle{plain}

\swapnumbers
\newtheorem{theorem}{Theorem}[section]
\newtheorem{proposition}[theorem]{Proposition}
\newtheorem{corollary}[theorem]{Corollary}
\newtheorem{lemma}[theorem]{Lemma}

\newtheorem{definition}[theorem]{Definition}

\newtheorem{claim}{Claim}

\newcommand{\la} {\lambda}

\newcommand{\R}{\mathbb{R}}

\newcommand{\RR}{{\mathbb R}}

\newcommand{\wt}{\widetilde}

\title[Partially Dominated Splittings]{Partially Dominated Splittings}

\thanks{
  L.S. is partially supported by FAPESB, CNPq and INCTMat-CAPES.
}

\author{Luciana Salgado}

\address{Instituto de Matem\'atica - Universidade Federal da Bahia
\\
Avenida Adhemar de Barros, s/n, Ondina, 40170110, Salvador, Bahia Brazil
}

 \email{lsalgado@ufba.br}

\date{\today}

\begin{document}

\begin{abstract}
 Let $\Lambda$ be a nonsingular compact invariant set for a $C^1$ flow $X_t$,
 defined over a compact riemannian manifold $M$.
  We want to know when the existence of a dominated splitting
  of the tangent bundle $T_{\Lambda}M$ on $\Lambda$ for the associated linear Poincar\'e flow
  is equivalent to the existence of a dominated splitting for the flow. For this,
  we propose a weak form of domination,
  called \emph{partially dominated splitting} and our main result is
  that there is a partially dominated splitting on $\Lambda$ for $X_t$
  if, and only if, the associated linear Poincar\'e flow has a dominated splitting.
\end{abstract}

\maketitle 

\section{Introduction and statement of results}\label{sec:intro}
\hfill

In order to prove the famous stability conjecture of Palis, many researchers proposed a several techniques, 
one of them the so called dominated splittings, as in \cite{Liao1981} and \cite{Ma1988}. 
In this setting, we have a decomposition of the tangent bundle over an invariant compact set for a diffeomorphism or a flow, 
into invariant subbundles which are related by the dynamics from a special form: the angle between the subbundles is far away from zero, 
when we iterate them by derivative of the diffeomorphism or flow (see, for instance, \cite{LiangLiu2008}).

Many works are related to the study of the dominated splittings and its connection with a several numbers of others dynamical phenomena 
such as homoclinic tangencies and robust transitivity, singular, sectional and nonuniform sectional hyperbolicity, among others, 
becoming relevant to study conditions under which a decomposition is dominated. 
see for instance \cite{AraArbSal2011, ArSal2012, ArSal2-2012, ArbSal,Moryiasu1991,Wen2002,NH2004,MPP04,BDV2004,BGV06,
    LiangLiu2008}. In particular, this notion is used for the Linear Poincar\'e Flow, 
    for instance in \cite{GW2006} in the case of absence of singularities and \cite{Doe1987, Liao1980,Man82}, in general.

It is very difficult to obtain domination for the flow from its associated linear Poincar\'e flow, 
even noting that this operator is only defined over regular orbits. In the singular setting, the authors in \cite{GLW2005}, 
have extended this notion and obtained a dominated splitting for the derivative of the flow over a robustly transitive set (allowing singularities) 
from a dominated splitting of the extended linear Poincar\'e flow.

Recall that, for a linear map $L: \R^n \to \R^n$ and a subset $F \subset \R^n \setminus \{0\}$, we have
\begin{equation}\label{eq1}
\Vert L \vert_{F}\Vert =  \sup_{\Vert v \Vert = 1, v \in F} \Vert L v \Vert,
\end{equation}
the norm of $L$ over $F$ and
\begin{equation}\label{eq2}
m(L \vert_{F}) =  \inf_{\Vert v \Vert = 1, v \in F} \Vert L v \Vert,
\end{equation}
the minimal norm of $L$ over $F$.
If $L$ is invertible and $F$ is a subspace of $\R^n$, then equation (\ref{eq2}) is equivalent to
\begin{equation}\label{eq2'}
m(L \vert_{F}) =  \Vert L^{-1} \vert_{L(F)}\Vert^{-1}.
\end{equation}

A compact invariant set $\Lambda$ is said to be hyperbolic for a flow $X_t$ if there exists  a
continuous invariant splitting of the tangent bundle over $\Lambda$, $T_\Lambda M= E^s\oplus E^X \oplus E^u$,
such that $E^s$ is uniformly contracted by $DX_t$ and $E^u$ is
uniformly expanded.

The usual definition of dominated splitting for a flow is established as follow.
\begin{definition}\label{def1}
  A \emph{dominated splitting} over a compact invariant set $\Lambda$ of $X$
  is a continuous $DX_t$-invariant splitting $T_{\Lambda}M =
  E \oplus F$ with $E_x \neq \{0\}$, $F_x \neq \{0\}$ for
  every $x \in \Lambda$ and such that there are positive
  constants $K, \lambda$ satisfying
  \begin{align}\label{eq:def-dom-split}
    \|DX_t|_{E_x}\|\cdot\|DX_{-t}|_{F_{X_t(x)}}\|<Ke^{-\la
      t}, \ \textrm{for all} \ x \in \Lambda, \ \textrm{and
      all} \,\,t> 0.
  \end{align}
\end{definition}

Analogous definition holds for diffeomorphisms.

A compact invariant set $\Lambda$ is said to be
\emph{partially hyperbolic} if it exhibits a dominated
splitting $T_{\Lambda}M = E \oplus F$ such that subbundle
$E$ is uniformly contracted, i.e. there exists $C>0$ and $\lambda>0$ such that $\|DX_t|_{E_x}\|\leq Ce^{-\lambda}$ for $t\geq 0$. 
In this case $F$ is called the \emph{central subbundle} of $\Lambda$.

Similarly, a compact invariant set $\Lambda$ is
\emph{volume hyperbolic} if it has a dominated splitting
$E\oplus F$ such that the volume is
uniformly contracted along $E$ and expanded along $F$ by
the action of the tangent cocyle. If the whole manifold $M$
is a volume-hyperbolic set for a flow $X_t$ (or a diffeomorphism), then we say
that $X_t$ is a volume-hyperbolic flow (diffeomorphism).




In what follows, we give the definition of linear Poincar\'e flow.

Let $\Lambda$ be an invariant nonsingular compact set for a $C^1$ flow $X_t$. 
Consider the quotient space by the flow direction of the tangent bundle on $\Lambda$, $\mathcal{N}^X = T_{\Lambda}M / \langle X \rangle$, 
which we may consider equivalent to the normal bundle to $\langle X \rangle$.

\begin{definition}\label{def-poinc-flow}
The linear Poincar\'e flow on $\Lambda$ associated to the flow $X_t$, $P_t^X : \mathcal{N}^X \to \mathcal{N}^X$ is defined by
\begin{eqnarray}
P_t^X = \Pi \circ DX_t,
\end{eqnarray}
where $\Pi: T_{\Lambda}M \to \mathcal{N}^X$ is the orthogonal projection on the normal bundle.
\end{definition}

A way to see if an invariant compact set $\Lambda$ without singularities has a hyperbolic splitting for a flow $X_t$ is verify if 
there is a hyperbolic splitting on $\Lambda$ for the associated linear Poincar\'e flow, see \cite{Doe1987} and \cite{Wojt2000}.

\begin{theorem}\label{equiv-hyp-split-flow}
     A nonsingular set $\Lambda$ is hyperbolic for a flow $X_t$ if, and only if, there is a hyperbolic splitting 
     on $\Lambda$ for the associated linear Poincar\'e flow.
\end{theorem}

It is proved below a similar statement for partial hyperbolicity over non-singular invariant sets for flows, see Corollary \ref{cor:ph-DD}.

It seems a natural question to ask if similar equivalence holds to dominated splittings. But, it is not true in general, 
as shows a counter example suggested by Pujals in \cite{BDV2004}. 
In this example, is obtained a robustly transitive suspended flow that has no dominated splitting.  
It is also known that every $C^1$-robustly transitive set for a diffeomorphism has a dominated splitting, see \cite{BDP2003}. 
On the other hand, Vivier in \cite{Viv2003}, proved that robustly transitive vector fields on a closed manifold of any dimension, 
always have a dominated splitting for the linear Poincar\'e flow. In the case of manifold with boundary, 
Gan, Li and Wen in \cite{GLW2005}, proved a singular version of Vivier's result by generalizing the notion of linear Poincar\'e flow 
and dominated splitting for singularities.

In fact, in \cite[Appendix B]{BDV2004} the authors exhibit diffeomorphisms on $\mathbb{T}^4$, robustly transitive, 
for which the finest dominated splitting $E \oplus F$ has neither contracting nor expanding subbundles 
(see figure \ref{fig:nonhypcomplex}): there are periodic orbits with contracting eigendirections contained in $F$ 
and other periodic orbits with expanding eigendirections contained in $E$. 
Thus, the suspension flow is a robustly transitive flow which has no dominated splitting. 
This is because the flow direction is dominated neither by $E$ nor $F$. 
Also see \cite{AraArbSal2011} for more studies about conditions for an invariant splitting for a flow to be dominated.

\begin{figure}[htpb]
    \includegraphics[height=2cm]{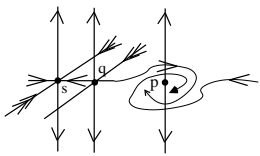}
    \caption{Saddles with real and complex eigenvalues.}
    \label{fig:nonhypcomplex}
\end{figure}


So, in order to study the domination properties and
the conditions for a splitting to be dominated for a $C^1$ flow, what sometimes
requires to know if some subbundle of the decomposition contains the field direction,
we propose here the following definition.

\begin{definition}\label{def-part-dom-flow}
  Let $\Lambda$ be a compact invariant set
  of a $C^1$ flow $X_t$.
  A \emph{partially dominated splitting} over $\Lambda$
  is a continuous $DX_t$-invariant splitting
  $T_{\Lambda}M = \wt{E} \oplus \langle X \rangle \oplus \wt{F}$
  where $\wt{E}_x \neq \{0\}$, $\wt{F}_x \neq \{0\}$ and
  $\langle X \rangle$ is the flow direction, for
  every $x \in \Lambda$ and such that there are positive
  constants $K, \lambda$ satisfying
  \begin{align}\label{eq:def-part-dom-split}
    \|DX_t|_{\wt{E}_x}\|\cdot\|DX_{-t}|_{\wt{F}_{X_t(x)}}\|<Ke^{-\la
      t}, \ \textrm{for all} \ x \in \Lambda, \ \textrm{and
      all} \,\,t> 0.
  \end{align}
\end{definition}

We are interested to study what are the conditions to a partial dominated splitting to induce a dominated splitting.

It is known that, in a hyperbolic decomposition $E^s_x \oplus \langle X \rangle \oplus E^u_x$ for a flow without singularities, 
both splittings $(E^s_x \oplus \langle X \rangle) \oplus E^u_x$ and $E^s_x \oplus (\langle X \rangle \oplus E^u_x)$ are dominated. 
This fact is a consequence of the uniforms expansion and contraction properties of $E^u_x, E^s_x$, respectively. 
On the other hand, if we have no uniform hyperbolicity or in the case of existence of hyperbolic singularities accumulated by regular orbits, 
to ensure that an invariant splitting is dominated we must know, among other things, in which subbundle is located the flow direction of the regular orbits.

Our main result is the following.

\begin{theorem}\label{th:poinc-flow-equivalence}
Let $X$ be a $C^1$-vector field on $M$ and $X_t$ its flow.
Let $\Lambda \subset M \setminus Sing(X)$ be a compact invariant set for $X_t$. Then,
there is a partially dominated splitting on $\Lambda$ for the flow if, and only if, the associated
Linear Poincar\'e Flow has a dominated splitting on $\Lambda$ .
\end{theorem}

This paper is organized as follows: in Section \ref{sec:intro}, we present main definitions, main result and an example as a motivation; 
in Section \ref{sec:applic}, we give some applications of the notion of partially dominated splitting and Theorem \ref{th:poinc-flow-equivalence}; 
in section \ref{sec:main-proof}, we have some complementary definitions and it is given the demonstration of the main theorem.

\section*{Acknowledgements}
The author would like to thank to V. Araujo, A. Arbieto, B. Santiago and C. Morales for very enlightening conversations. 
She also thanks the financial support of FAPESB-Jovem Cientista, CNPq and INCTMat-CAPES postdoctoral scholarship at IMPA-Brazil, 
where this work began to be developed.

\section{Applications}\label{sec:applic}

In the sequence, we present some applications.

First, we recall the following result of \cite{AraArbSal2011}.

Let $\Lambda$ be a non-singular compact invariant set for the flow
$X_t$ of a $C^1$ vector field $X$ on $M$.
\begin{lemma}
  \label{le:flow-center}
  Given a continuous invariant splitting
  $E\oplus F$ of $T_\Lambda M$ over
  $\Lambda$, such that $E$ is uniformly
  contracted, then the flow direction is contained in the
  $F$ subbundle, for all $x\in \Lambda$.
\end{lemma}

Note that the hypothesis about the angle might be changed by continuity of the splitting.

One of the main features of a dominated splitting $E \oplus F$ is that $E$ may not be contracting and $F$ may not be expanding, 
but $E$ must be contracting if $F$ is not expanding and $F$ must be expanding if $E$ is not contracting.

Based on this fact, we have the following.

\begin{proposition}\label{cor:suf-cond-DD}
A sufficient condition for a partially dominated splitting $E\oplus \langle X \rangle \oplus F$ over a compact invariant nonsingular set for a flow 
to induce a dominated one is that the subbundle $E$ (or $F$) is either uniformly contracting or uniformly expanding.
\end{proposition}

\begin{proof}
First, suppose that $E$ is an uniformly contracting subbundle, i.e., there is $K>0$ such that
\begin{align*}
\Vert DX_t\vert_{E_x}\Vert \leq K e^{-\lambda t}, \ \forall t > 0, \ \forall x \in \Lambda.
\end{align*}

Then, taking $G:= \langle X \rangle \oplus F$, we must have given $u \in E_x$ and $v = \alpha X(x) + v^F \in G_x, v^F \in F, \alpha \in \mathbb{R}$,
\begin{align*}
\Vert DX_t(x)u \Vert  \cdot \Vert DX_{-t}(x) v \Vert = \\
\Vert DX_t(x)u \Vert  \cdot \Vert DX_{-t}(x) \alpha X(x) + DX_{-t}(x) v^F \Vert \\
\leq \Vert DX_t(x)u \Vert  \cdot \big[\Vert  \alpha X(X_{-t}(x)) \Vert +\Vert DX_{-t}(x) v^F \Vert \big]  \\
\leq K e^{-\lambda t} \Vert  \alpha X(X_{-t}(x)) \Vert + \Vert DX_t(x)u \Vert  \cdot \Vert DX_{-t}(x) v^F \Vert\\
\leq  C e^{-\lambda t}, \ \forall t > 0, \ \forall x \in \Lambda.
\end{align*}
The last inequality is true, because there exist constants $A, B > 0$ such that $A \leq \Vert X(X_{-t}(x))\Vert \leq B$ 
and we have the domination condition between $E$ and $F$.

On the other hand, look that, if $F$ is uniformly expanding, by the same argument applied to the reverse flow, we obtain a similar result. 
This complete the proof.
\end{proof}

\begin{corollary}\label{cor:ph-DD}
A compact invariant nonsingular set $\Lambda$ is partially hyperbolic for a flow $X_t$ if, and only if, 
the associated Linear Poincar\'e Flow is partially hyperbolic on $\Lambda$.
\end{corollary}

\begin{proof}
First, suppose that the Linear Poincar\'e flow $P_t$ has a partial hyperbolic splitting $\tilde{E} \oplus \tilde{F}$ on $\Lambda$, 
with $\tilde{E}$ uniformly contracting. In particular, this is a dominated splitting for $P_t$. So, by Theorem \ref{th:poinc-flow-equivalence}, 
we obtain a partially dominated splitting $E \oplus \langle X \rangle \oplus F$ for the flow $X_t$ on $\Lambda$. 
Therefore, by Proposition \ref{cor:suf-cond-DD} this splitting induces a dominated one, once $E$ is uniformly contracting and dominates $F$. 
Since $\Lambda$ has no singularities, the splitting $E \oplus E^c$ is partially hiperbolic for $X_t$, where $E^c := \langle X \rangle \oplus F$.

Conversely, suppose that $E \oplus F$ is a partial hyperbolic splitting for $X_t$ on $\Lambda$, with $E$ an uniformly contracting subbundle. 
Then, by Lemma \ref{le:flow-center}, the flow direction is contained into subbundle $F$.

By writting $\tilde{F} := F\setminus\langle X \rangle$, we have $E \oplus F = E \oplus \langle X \rangle \oplus \tilde{F}$. 
We affirm that this is a partially dominated splitting.

Suppose, by contradiction, that the subbundle $E$ does not dominate $\tilde{F}$.
Note that $\Lambda$ is compact, has no singularities, $E$ is uniformly contracting and $\langle X \rangle$ does not vanish. 
Thus, if $\tilde{F}$ contracts sharply than $E$ for $DX_t$, we must have $DX_{-t}\vert_{\tilde{F}}$ expanding more sharply than $DX_t\vert_{E}$ contracts. 
Then, taking a non zero vector $v = v^X \oplus v^{\tilde{F}}$, such that $0 \neq v^X \in \langle X \rangle$ and $v^{\tilde{F}} \in \tilde{F}$, 
we must have for some $t_0 > 0$
\begin{align*}
1 < \Vert DX_{t_0}\vert_{E} \Vert \cdot \Vert DX_{-t_0} v^{\tilde{F}}\Vert
\leq
\Vert DX_{t_0}\vert_{E}  \Vert \cdot \Vert DX_{-t_0} (v^X + v^{\tilde{F}})\Vert
\leq \frac{1}{2}.
\end{align*}
But, it is a contradiction. So, we are done.
\end{proof}

\section{Proof of Theorem \ref{th:poinc-flow-equivalence}}\label{sec:main-proof}

In this section, we give the demonstration of the Theorem $\ref{th:poinc-flow-equivalence}$. 
Before this, we give some more definitions which will be useful in the demonstration.

In \cite{NH2004}, Newhouse has shown some conditions for dominated and hyperbolic splittings on compact invariant sets 
for diffeomorphisms based on induced action on cones fields and its complement. In what follows, it is used his terminology.

Let $E \subset \RR^n$ be a proper subspace, i.e, $0 < dim E < n$. Let $F$ be a complementary subspace, i.e.,
$\RR^n = E \oplus F$.
The \emph{standard unit cone} determined by the subspaces $E$ and $F$ is the set
$$
C_1(E,F)= \{ u = (u_1, u_2), u_1 \in E, u_2 \in F,  \vert v_2\vert \leq \vert v_1\vert \}.
$$

A \emph{cone} $C(E)$, with core $E$, in $\RR^n$ is the image of $C_1(E,F)$ by an linear automorphism $T:\RR^n \to \RR^n$ such that $T(E)=E$. 
A cone $C$ in $\RR^n$ is a set $C(E)$, where $E$ is a proper subspace.

Let $X_t$ be a $C^1$ flow generate by a vector field $X$ on $M$ and $\Lambda$ be a compact invariant set of $X$.

A \emph{cone field} $C = \{C_x\}$ on $\Lambda$ is a collection of cones $C_x \subset T_xM$, for $x \in \Lambda$. 
A cone field is said to have \emph{constant orbit core dimension} if $dim E_x = dim E_{X_t(x)}$ for all $x \in \Lambda$, 
where $E_x$ and $E_{X_t(x)}$ are the core of $C_x, C_{X_t(x)}$, respectively.

Given a cone field $C = \{C_x\}, x \in M$, let
\begin{align*}
m_{C,x}=m_{C,x}(X)= \inf_{u \in C_x \setminus\{0\}} \frac{\Vert DX_t(x) u\Vert}{\Vert u\Vert}
\end{align*}
and
\begin{align*}
{m'}_{C,x}={m'}_{C,x}(X)= \inf_{u \notin C_{X_t(x)}} \frac{\Vert DX_{-t}(X_t(x)) u\Vert}{\Vert u\Vert},
\end{align*}
be the minimal expansion and the minimal co-expansion of $DX_t$ on $C_x$, respectively.
\begin{definition}\label{def-exp-cone}
The domination coefficient of $DX_t$ on $C$ is
\begin{align}
m_d (C) = m_d (C,X) = \inf_{x \in \Lambda} m_{C,x} \cdot {m'}_{C,x}.
\end{align}
If $C$ has constant orbit core dimension and $m_d(C) > 1$, one say that $X$ is \emph{dominating} on $C$ over $\Lambda$. 
Moreover, we say that $X$ is strongly dominating on $C$ if $C$ has constant orbit core dimension and
\begin{align}
\big(\inf_{x \in \Lambda} m_{C,x} \big) \cdot \big(\inf_{x \in \Lambda} {m'}_{C,x}\big) > 1.
\end{align}
\end{definition}

The next useful result is a flow version of another one due to Newhouse in \cite{NH2004}.

\begin{proposition}\cite[Proposition 1.3-cocycles version]{NH2004}
\label{prop:NH}
A sufficient condition for a differentiable cocycle $A_t$ over a vector bundle with compact base $\Lambda$ to have a dominated splitting  
is that there is an $t_0 \in \RR^+$ such that $A_{t_0}$ has a strongly dominated cone field over $\Lambda$.
\end{proposition}

Now, we are able to give the demonstration of our main theorem.

\begin{proof}[\bf Proof of Theorem \ref{th:poinc-flow-equivalence}]
\hfill

Let $T_{\Lambda}M = E \oplus \langle X \rangle \oplus F$ a partially dominated splitting for $X_t$ on $\Lambda$, i.e.,
there are constants $K > 0$, $0 < \lambda < 1$ such that
$$
\Vert DX_t\vert_{E_x}\Vert \cdot \Vert DX_{-t}\vert_{F_{X_t(x)}}\Vert < K e^{-\lambda t}.
$$

By making a change of coordinates, if necessary, we may consider that the three subbundles are orthogonal between them. 
In fact, take some normalized basis
$$
\{ e_i(x) \}_{i=1}^{l}, u(x), \{ f_j(x) \}_{j=1}^{k}
$$
of $E, \langle X(x) \rangle$ and $F$, respectively. Hence, any vector $z \in T_xM$ can be written as
$$
z = \sum_{i=1}^{l}\alpha_i(x) e_i(x) + \gamma(x) u(x) + \sum_{j=1}^{k} \beta_j(x) f_j(x),
$$
with $\alpha_i(x), \gamma(x), \beta_j(x) \in \mathbb{R}$, $\dim E = l, \dim F = k$ and $l + k + 1 = \dim (T_xM)$. 
Now, we define a new inner product $[\cdot , \cdot]$ given by
\begin{align*}
[z_1, z_2] = \sum_{i=1}^{l}\alpha_i(x) + \gamma(x) + \sum_{j=1}^{k} \beta_j(x),
\end{align*}
where
$$
z_1 = \alpha_{1,i}(x) e_i(x) + \gamma_1(x) u + \sum_{j=1}^{k} \beta_{1,j}(x) f_j(x)
$$
and
$$
z_1 = \alpha_{2,i}(x) e_i(x) + \gamma_2(x) u + \sum_{j=1}^{k} \beta_{2,j}(x) f_j(x).
$$

So, in the induced metric by $[\cdot , \cdot]$, which we denote $| \cdot |$, we have
\begin{align*}
|z|^2 = \sum_{i=1}^{l}(\alpha_i(x))^2 + (\gamma(x))^2 + (\sum_{j=1}^{k} \beta_j(x))^2
\end{align*}
and $e_i(x), u(x), f_j(x)$ are mutually orthogonal, i.e.,
$$
[e_i(x),f_j(x)] = 0, [e_i(x), u(x)] = 0 \ \textrm{and} \ [f_j(x), u(x)] = 0, \forall i=1,\cdots l, j=1,\cdots, k.
$$
Note that this process is continuous with respect to $x \in M$, because the splitting $E \oplus \langle X \rangle \oplus F$ is  continuous.

Thus, as we only replace the initial metric for an equivalent one, with the same continuous invariant subspaces of $T_xM$ for all $x \in M$, 
we may consider $E \oplus F$ as the normal bundle of $\langle X \rangle$, $N^X$, and by writing $E = N^-_{\Lambda}, F = N^+_{\Lambda}$ 
we obtain a dominated splitting
$$
N^X = N^-_{\Lambda} \oplus N^+_{\Lambda},
$$
for $P^X_t$.

Conversely, suppose that there exists a dominated splitting
$N^X = N^-_{\Lambda} \oplus N^+_{\Lambda}$ for $P^X_t$.
We need reconstruct the subspaces of partially dominated splitting
from its correspondent projections on $\wt{T}_{\Lambda}M$.

Define the the following subbundles:
$$
A_{\Lambda}:= N^-_{\Lambda} + \langle X \rangle \ \textrm{and} \ B_{\Lambda}:= \langle X \rangle + N^-_{\Lambda},
$$
over $\Lambda$.

\begin{claim}
The subbundles $A_{\Lambda}$ and $B_{\Lambda}$ are $DX_t$-invariant.
\end{claim}
Indeed, by following \cite[Lemma 2.5]{BM2010}, take $x \in \Lambda$ and $v_x \in A_x$, 
so there is a unique $v_x^- \in N^-_x$ such that $v_x - v_x^-  \in \langle X(x) \rangle$. 
As $\langle X(x) \rangle$ is $DX_t$-invariant, for $t \in \RR$, so
$$
DX_t(x)v_x - DX_t(x)v_x^- \in \langle X(X_t(x)) \rangle.
$$
Hence, $\Pi_{X_t(x)}DX_t(x)v_x = P_t(x)v_x^-$ and, from this, we have
\begin{align*}
\Pi_{X_t(x)}DX_t(x)v_x \in N^-_{X_t(x)},
\end{align*}
once $N^-_{\Lambda}$ is $P_t$-invariant.

Since, by definition of the orthogonal projection, $DX_t(x)v_x - \Pi_{X_t(x)}DX_t(x)v_x \in \langle X(x) \rangle$, we obtain,
\begin{align*}
DX_t(x)v_x \in N^-_{X_t(x)} + \langle X(X_t(x)) \rangle = A_{X_t(x)},
\end{align*}
so, $A_{\Lambda}$ is $DX_t$-invariant. Analogously for $B_{\Lambda}$. Thus, we have proved the claim.

Now, since $\Lambda$ is non-singular and compact, we can take cone fields around $N^-_{\Lambda}$ and $N^+_{\Lambda}$ which are, respectively, 
complementary to $\langle X \rangle$ obtaining two complementary subbundles to 
$T_{\Lambda}M \setminus \langle X \rangle$, $\widetilde{E}$ and $\widetilde{F}$ for which hold the domination condition.

We get the above mentioned by using a flow version of Newhouse's result.

In fact, take a cone $C^+_x$ with core $N^+_x$ in $\hat{T}_xM := T_xM/\langle X\rangle$. Note that, 
because $\langle X\rangle$ is a $DX_t$-invariant subspace, we have its complementary in $T_xM$ is also invariant. 
So, we can take cones into this complementary subspace and consider if we have strongly domination there. 
We affirm that $DX_t(x)$ is strongly dominating on $C^+_x$ into $T_{\Lambda}M \setminus \langle X \rangle$.

Indeed, because the domination condition on $N^-_{\Lambda}$ and $N^+_{\Lambda}$ for the linear Poincar\'e flow $P_t$, we must have, for some $T_0 > 0$,
\begin{align*}
\big(\inf_{x \in \Lambda} m_{C^+_x,x}(DX_{T_0}) \big) \cdot \big(\inf_{x \in \Lambda} {m'}_{C^+_x,x}(DX_{T_0})\big) \geq \\
\big(\inf_{x \in \Lambda} m_{C^+_x,x}(P_{T_0}) \big) \cdot \big(\inf_{x \in \Lambda} {m'}_{C^+_x,x}(P_{T_0})\big) > 1.
\end{align*}
Look that the first inequality is also true: by definition, $P_t = \Pi \circ DX_t$ and $\Vert \Pi \Vert = 1$, so
\begin{align*}
\Vert P_t v \Vert = \Vert \Pi \circ DX_t v \Vert \leq \Vert \Pi \Vert \cdot \Vert DX_t v \Vert = \Vert DX_t v \Vert.
\end{align*}

We also note that, for a fixed $t>0$, the limit subspace
\begin{align*}
C^+_n = \cap_{n \geq 0} DX_{-nt}(X_{-nt}(x))(C^+_{X_{-nt}(x)}), n \in \mathbb{N},
\end{align*}
is $DX_t$-invariant. Moreover, it cannot contain the flow direction, otherwise we must have
\begin{align*}
DX_{nt} (C^+_n(x)) =\\
DX_{nt} \big[ \cap_{n \geq 0} DX_{-nt}(X_{-nt}(x))(C^+_{X_{-nt}(x)}) \big] \subset\\
\cap_{n \geq 0} DX_{nt} (DX_{-nt}(X_{-nt}(x))(C^+_{X_{-nt}(x)})) = \\
C^+_{X_{-nt}(x)} \subset N^-_{X_{-nt}(x)} \oplus N^+_{X_{-nt}(x)}.
\end{align*}
But, by invariance of $\langle X \rangle$ and because it is orthogonal to $N^- \oplus N^+$, so there exists a singularity. 
This contradiction shows that $\langle X(x) \rangle \nsubseteq C^+_n(x)$. And, finally, our claim is true.

So, by Proposition \ref{prop:NH}, we obtain a dominated splitting $\widetilde{E} \oplus \widetilde{F}$ for $DX_t$ over 
$T_{\Lambda}M \setminus \langle X \rangle$.

Thus,
\begin{align*}
T_{\Lambda}M = \widetilde{E} \oplus \langle X \rangle \oplus \widetilde{F},
\end{align*}
is a partially dominated decomposition over $\Lambda$.

\end{proof}

\end{document}